\documentclass{article}
\usepackage{amsmath,amssymb,amsthm}
\title{A bounded jump for the bounded Turing degrees}

\author{Bernard A. Anderson\\ \small{Department of Mathematics}\\ \small{and Natural Sciences}\\ \small{Gordon College}\\ \small{banderson@gdn.edu}\\ \small{www.gdn.edu/Faculty/banderson} \and
Barbara F.~Csima \thanks{B.\ Csima was partially supported by Canadian NSERC
Discovery Grant 312501. B. Csima would like to thank the Max Planck Institute for Mathematics, Bonn Germany, for a productive visit.} \\ \small{Department of}\\ \small{Pure Mathematics}\\ \small{University of
Waterloo}\\
\small{csima@math.uwaterloo.ca}\\
\small{www.math.uwaterloo.ca/$\sim$csima}}

\newtheorem{theorem}{Theorem}[section]
\newtheorem{lemma}[theorem]{Lemma}
\newtheorem{cor}[theorem]{Corollary}
\newtheorem{prop}[theorem]{Proposition}
\newtheorem{remark}[theorem]{Remark}

\theoremstyle{definition}
\newtheorem{definition}{Definition}

\newcommand{\setsep}{\ensuremath{\, | \;}}

\newcommand{\strings}{2^{<\omega}}
\newcommand{\num}{\in\omega}

\newcommand{\conv}{\!\downarrow}
\newcommand{\dive}{\!\uparrow}

\newcommand{\concat}{\: \hat{\ } \:}
\newcommand{\rstrd}{\upharpoonright\!\!\upharpoonright}
\newcommand{\smrstrd}{\upharpoonright\!\!\upharpoonright}

\newcommand{\btl}{\leq_{bT}}
\newcommand{\es}{\emptyset}
\newcommand{\ce}{c.e.\ }
\newcommand{\la}{\langle}
\newcommand{\ra}{\rangle}
\newcommand{\Eres}{\upharpoonright\!\!\upharpoonright}
\newcommand{\dn}{\downarrow}
\newcommand{\up}{\uparrow}
\newcommand{\nin}{\not\in}
\newcommand{\rt}{\rightarrow}

\begin{document}
\maketitle
\begin{abstract} We define the bounded jump of $A$ by
$A^b = \{ x \num \setsep \exists i \leq x [\varphi_i (x) \conv
\ \wedge\ \Phi_x^{A \smrstrd \varphi_i(x)}(x)\conv ] \}$ and
let $A^{nb}$ denote the $n$-th bounded jump. We demonstrate
several properties of the bounded jump, including that it is
strictly increasing and order preserving on the bounded Turing
($bT$) degrees (also known as the weak truth-table degrees). We
show that the bounded jump is related to the Ershov hierarchy.
Indeed, for $n \geq 2$ we have $X \leq_{bT} \es^{nb} \iff X$ is
$\omega^n$-c.e.\ $\iff X \leq_1 \es^{nb}$, extending the
classical result that $X \leq_{bT} \es' \iff X$ is
$\omega$-c.e.  Finally, we prove that the analogue of
Shoenfield inversion holds for the bounded jump on the bounded
Turing degrees.  That is, for every $X$ such that $\es^b
\leq_{bT} X \leq_{bT} \es^{2b}$, there is a $Y \leq_{bT} \es^b$
such that $Y^b \equiv_{bT} X$.
\end{abstract}

\section{Introduction}

In computability theory, we are interested in comparing the
relative computational complexities of infinite sets of natural
numbers. There are many ways of doing this, and which method is
used often depends on the purpose of the study, or how fine a
comparison is desired.  Two sets of the same computational
complexity ($X \leq Y$ and $Y \leq X$) are said to be in the
same \emph{degree}.  The computable sets form the lowest degree
for all of the reducibilities we consider here.

Some of the most natural reducibilities are $m$-reducibility
and $1$-reducibility.  Recall that a set $A$ is $m$-reducible
($1$-reducible) to a set $B$ if there is a computable
(injective) function $f$ such that for all $x$, $x \in A$ iff
$f(x) \in B$. The major failing of these reducibilities is that
a set need not be reducible to its complement.

The most commonly studied reducibility is that of Turing
reducibility, where $A$ is Turing reducible to $B$ if there is
a program that, with reference to an infinite oracle tape
containing $B$, computes $A$. Though each computation of a
Turing reduction views only finitely much of the oracle tape,
there is no computable bound on how much of the tape can be
viewed in a computation.

Many natural Turing reductions have the property that the use
of the oracle is bounded by a computable function. We will
refer to such reductions as \emph{bounded Turing} reductions,
and write $A \leq_{bT} B$. This is also commonly known in the
literature as weak truth-table ($wtt$) reducibility.

A truth-table reduction is a pair of computable functions $f$
and $g$, such that, for each $x$, $f(x)$ supplies a finite list
$x_1,..., x_n$ of positions of the oracle, and $g(x)$ provides
a truth-table on $n$ variables (a map $2^n \rt 2$). A set $A$
is said to be truth-table reducible to $B$ if there is a
truth-table reduction $f, g$ such that, for every $x$, $x \in
A$ iff the row of the truth table $g$ obtained by viewing $B$
on the positions $x_1,...,x_n$ has value $1$. It is easy to see
that $A \leq_{tt} B$ iff $A$ is Turing reducible to $B$ via a
functional that is total on all oracles. Note that if a
functional is total on all oracles, then there is a computable
bound on the use for each input.  Bounded Turing reducibility is weaker than
$tt$-reducibility, and this is where the name ``weak
truth-table reducibility" originated. However, since the
weakening has nothing to do with the truth-table, we follow the
notation of $bT$, as used in Soare \cite{NewSoare} \cite{Soare}.

The halting set is the first natural example of a
non-computable set. The Turing jump operator works by
relativizing the halting set to other oracles. Basic properties
of the Turing jump include that it is strictly increasing with
respect to Turing reducibility, and that it maps a single
Turing degree into a single 1-degree. This later property shows
that the Turing jump is a well-defined operator on all of the
degree structures we have mentioned so far.

The strictly increasing property of the Turing jump implies
that the Turing jump of any set must compute the halting set.
There are a variety of ``jump inversion'' results, that show
that the range of the Turing jump is maximal (with respect to a
restricted domain). Friedberg jump inversion states that for
every $X \geq_T \es'$ there exists $A$ with $A' \equiv_T X
\equiv_T A \oplus \es'$. Shoenfield \cite{ShoenfieldInv}
demonstrated that for every $\Sigma_2$ set $X \geq_T
\es^\prime$ there is a set $Y \leq_T \es^\prime$ such that
$Y^\prime \equiv_T X$.

What about jump inversion for strong reducibilities? Mohrherr
\cite{Mohrherr} showed that for any $X \geq_{tt} \es'$, there
exists $A$ with $A' \equiv X$. Anderson \cite{andtt} showed
that the full analogue of Friedberg jump inversion holds: for
every $X \geq_{tt} \es'$, there exists $A$ with $A' \equiv_{tt}
X \equiv_{tt} A \oplus \es'$. Both Mohrherr's and Anderson's
proofs work with bT in place of tt. However, Csima, Downey, and
Ng \cite{CsimaDowneyNg} have proved that the analogue of
Shoenfield jump inversion fails to hold for the tt and bounded
Turing degrees. Indeed, they showed that there is a $\Sigma_2$
set $C
>_{tt} \es^\prime$ such that for every $D \leq_T \es^\prime$ we
have $D^\prime \not\equiv_{bT} C$. The proof exploits the fact
that the Turing jump is defined with respect to Turing (and not
bounded Turing) reducibilities.

Our goal for this paper was to develop a jump operator for the
bounded Turing degrees.  We wanted this jump to be bounded in
its use of the oracle, and to hold all of the properties
usually associated with a jump operator (in particular,
strictly increasing and order preserving).  In this paper, we
will define such a jump, examine its properties, and show it is
distinct from already used operators.  We will prove Shoenfield
inversion holds for the bounded Turing degrees with this jump.

The sets that are computable from the $n$-th Turing jump of
$\es$ have a very nice characterization -- they are exactly the
$\Delta^0_{n+1}$ sets. The $n$-th bounded jumps of $\es$ also
have a natural characterization. In this case, the connection is
with the Ershov hierarchy.  For $n \geq 2$, the sets that can be bT
computed from (indeed are tt or 1-below) the $n$-th iterate of the
bounded jump are exactly the $\omega^n$-c.e. sets.

There have been other jumps for strong reducibilities
introduced in the past, and we discuss some of these in Section
\ref{sec:other jump operators}. It has recently come to our
attention that Coles, Downey, and Laforte \cite{CDL} had
studied an operator similar to our bounded jump (defined as
$A^{b_1}$ in this paper), but unfortunately no written record
of their work exists, beyond a proof that their jump is
strictly increasing.

\section{Notation}

We mainly follow the standard notation for computability theory
as found in Cooper \cite{CooperBook} and Soare \cite{NewSoare}
\cite{Soare}. We let $\varphi_0, \varphi_1, \varphi_2, ... $
be an effective enumeration of the partial computable functions,
and let $\Phi_0, \Phi_1, \Phi_2,... $ be an effective enumeration
of the Turing functionals.  We assume our enumerations are acceptable.

We let $\es' = \{x \mid \varphi_x(x)\dn \}$, and for an
arbitrary set $A$, let $A' = \{ x \mid \Phi_x^A(x) \dn \}$. In
the case that the enumeration $\{\varphi_n\}_{n\in \omega}$ is
such that $\varphi_n = \Phi_n^{\es}$, then there is no
confusion with the two definitions of $\es'$. But under any
enumeration, the two definitions are $1$-equivalent.

For a set $A$, we let $A \Eres x = \{n \in A \mid n \leq x\}$.
We follow an expression with a stage number in brackets
(i.e.\ $[s]$) to indicate the stage number applies to everything
in the expression that is indexed by stage.

For sets $A$ and $B$ we write that $A \btl B$, and say $A$
is bounded Turing reducible to $B$, if there exist $i$ and $j$
such that $\varphi_j$ is total and for all $x$, $A(x) =
\Phi_i^{B\Eres \varphi_j (x)}(x)\dn$. This agrees with the
informal definition of bT given in the introduction.

\section{The bounded jump}

\begin{definition} Given a set $A$ we define the bounded jump
$$A^b = \{ x \num \setsep \exists i \leq x [\varphi_i (x)
\conv \ \wedge\ \Phi_x^{A \smrstrd \varphi_i(x)}(x)\conv ] \}$$ \end{definition}

We let $A^{nb}$ denote the $n$-th bounded jump.

\begin{remark} $\es^b \equiv_1 \es^\prime$ \end{remark}

This holds since bounding the use of an empty oracle has no effect.  We will use $\es^b$ and $\es^\prime$ interchangeably from now on.

We consider a more general definition of the bounded jump.

\begin{definition} $A^{b_0} = \{ \langle e, i, j \rangle \num \setsep \varphi_i (j) \conv \ \wedge\ \Phi_e^{A \smrstrd \varphi_i (j)} (j) \conv \}$ \end{definition}

We show that, up to truth table equivalence, $A^b$ and $A^{b_0}$ are the same.  We will at times identify one with the other.

\begin{remark} For any set $A$ we have $A^b \leq_{tt} A^{b_0}$. \end{remark}

This is true since $x \in A^b \iff \exists i \leq x [ \langle
x, i, x \rangle \in A^{b_0} ]$.

\begin{theorem} For any set $A$ we have $A^{b_0} \leq_1 A^b$. \end{theorem}
\begin{proof} We define a computable and injective function $k$ by $\varphi_{k(i,j)} (x) = \varphi_i (j)$.  We then define the function $g$ (also computable and injective) by
$$\Phi^C_{g(\langle e, i, j \rangle )} (x) = \begin{cases} \Phi_e^{C \smrstrd \varphi_{k(i,j)}(x)} (j) & \; \varphi_i (j) \conv \\ \dive & \; \text{else} \end{cases}$$

By the padding lemma we may assume without loss of generality that for all $e$, $i$, and $j$ we have $g( \langle e, i, j \rangle ) \geq k (i,j)$.

We now show that $\langle e, i, j \rangle \in A^{b_0} \iff g(
\langle e, i, j \rangle ) \in A^b$.

For the forward direction, we use $k(i,j)$ as the witness that
$g(\langle e, i, j \rangle) \in A^b$.

\begin{align*}
& \langle e, i, j \rangle \in A^{b_0} \\
 \Rightarrow \ & \varphi_i (j) \conv \mbox{ and } \Phi_e^{A \smrstrd \varphi_i (j)} (j) \conv \\
 \Rightarrow \ & \mbox{for any $x$, } \varphi_{k(i,j)}(x) \conv \mbox{ and }
   \Phi_{g(\langle e, i, j \rangle)}^{A \smrstrd \varphi_{k(i,j)}(x)} (x) \conv
   \ \mbox{[by definitions of $g$ and $k$]} \\
 \Rightarrow \ & \varphi_{k(i,j)}(g(\langle e, i, j \rangle)) \conv \mbox{ and }
    \Phi_{g(\langle e, i, j \rangle)}^{A \smrstrd \varphi_{k(i,j)}(g(\langle e, i, j \rangle))}
    (g(\langle e, i, j \rangle)) \conv \ [\mbox{let } x = g(\langle e, i, j \rangle)] \\
 \Rightarrow \ &  \exists l \leq g(\langle e, i, j \rangle) \ [\, \varphi_l (g(\langle e, i, j \rangle))
     \conv \mbox{ and } \Phi_{g(\langle e, i, j \rangle)}^{A \smrstrd \varphi_l
     (g(\langle e, i, j \rangle))} (g(\langle e, i, j \rangle)) \conv \,]\ [\mbox{let } l = k(i,j)] \\
 \Rightarrow \ & g(\langle e, i, j \rangle) \in A^b
\end{align*}

For the backward direction, we ignore the witness $l$ that
$g(\langle e, i, j \rangle) \in A^b$, and rely on the
definition of $g$.

\begin{align*}
& g(\langle e, i, j \rangle) \in A^b\\
\Rightarrow \ & \exists l \leq g(\langle e, i, j \rangle) \ [\, \varphi_l
(g(\langle e, i, j \rangle)) \conv \mbox{ and } \Phi_{g(\langle
e, i, j \rangle)}^{A \smrstrd \varphi_l (g(\langle e, i, j
\rangle))} (g(\langle e, i, j \rangle)) \conv \,] \\
\Rightarrow \ & \exists l \leq g(\langle e, i, j \rangle) \ [\, \varphi_l (g(\langle e, i, j \rangle))
     \conv \mbox{ and } \varphi_i (j) \conv \mbox{ and } \Phi_e^{\big(A \smrstrd
     \varphi_l (g(\langle e, i, j \rangle))\big) \smrstrd \varphi_{k(i,j)}(g(\langle e, i, j \rangle))} (j)
      \conv \,] \\
 & \ \ \ \ \ \ \      \mbox{ [by definition of $g$]}\\
\Rightarrow \ & \exists l \leq g(\langle e, i, j \rangle) \ [\, \varphi_l (g(\langle e, i, j \rangle)) \conv
    \mbox{ and } \varphi_i (j) \conv \mbox{ and } \Phi_e^{A \smrstrd \min \big( \varphi_l
    (g(\langle e, i, j \rangle)), \varphi_i (j) \big)} (j) \conv \,]\\
\Rightarrow \ & \varphi_i (j) \conv \mbox{ and } \Phi_e^{A \smrstrd \varphi_i
(j)} (j) \conv \\
\Rightarrow \ & \langle e, i, j \rangle \in A^{b_0}
\end{align*}

\end{proof}

We see later in Remark \ref{notoneequiv} that we cannot strengthen this to $A^b \equiv_1 A^{b_0}$.

Another possibility is a more ``diagonal'' definition for the bounded jump.

\begin{definition} $A^{b_1} = \{x \setsep \varphi_x (x) \conv \ \wedge\ \Phi_x^{A \smrstrd \varphi_x (x)} (x) \conv \}$ \end{definition}

We view this definition as less desirable, since it depends
heavily on the particular enumeration $\{\varphi_x\}_{x \in
\omega}$ of the partial computable functions. Indeed, depending
on the enumeration, one could have $A^{b_1} = \es'$ for all
sets $A$, or with a different enumeration, $A^{b_1} \equiv_1
A^{b_0}$.

%
%
%
%
%
%
%
%

Finally, we might also consider a simpler bounded jump.

\begin{definition} $A^i = \{ x \num \setsep \Phi_x^{A \smrstrd x} (x) \conv \}$ \end{definition}

However, this definition seems unsatisfactory since it is not strictly increasing.

\begin{remark} Let $A$ be a set with $A \geq_{bT} \es^\prime$.  Then $A \geq_{bT} A^i$. \end{remark}
\begin{proof} We show that $A^i \leq_{bT} A \oplus \es^\prime$ for any $A$.  Let $f(n)$ denote the maximum over all strings $\sigma$ of length $n$, of the location of $\es^\prime$ needed to determine if $\Phi_n^\sigma (n) \conv$.  Then $\es^\prime \rstrd f(n)$ and $A \rstrd n$ suffice to compute $A^i (n)$.  \end{proof}

\section{Properties} \label{Properties}
We summarize some facts about the bounded jump.  Let $A$ be any set.

\begin{enumerate}
\item{$\es^b \equiv_1 \es^\prime$.}
\item{$A \leq_1 A^b$.}
\item{$A^b \leq_1 A^\prime$ (since $A^b$ is \ce in $A$).}
\item{$\es^\prime \leq_{tt} A^b$ (as a consequence of Corollary \ref{ordprev} below).}
\item{$A^b \equiv_T A \oplus \es^\prime$ (by Proposition \ref{Turchar} below).}
\item{If $A \geq_T \es^\prime$ then $A^b \leq_T A$.}
\item{Let $A$ be such that $A^\prime \not \leq_T A \oplus \es^\prime$ (e.g.\ any $A \geq_T \es^\prime$).  Then $A^\prime \not \leq_T A^b$ (so $A^\prime \not \leq_{bT} A^b$).}
\item{$A^b \not \leq_{bT} A$ (by Theorem \ref{strinc} below).}
\item{If $A \geq_{bT} \es^\prime$ then $A^b \not \leq_{bT} A \oplus \es^\prime$.}
\end{enumerate}

The effect of the bounded jump on the Turing degrees is easy to characterize.

\begin{prop} Let $A$ be any set.  Then $A^b \leq_T A \oplus \es^\prime$. \label{Turchar} \end{prop}
\begin{proof} We wish to determine if a given $n$ is such that $\exists i \leq n \, [\varphi_i (n) \conv \ \wedge\\ \Phi_n^{A \smrstrd \varphi_i(n)}(n)\conv ]$.  We note the existential quantifier is bounded.  Given $i \leq n$, we use $\es^\prime$ to determine if $\varphi_i (n) \conv$.  If it does we then get $\sigma = A \rstrd \varphi_i (n)$ from $A$ and use $\es^\prime$ to determine if $\Phi_n^\sigma (n) \conv$.  This does not require $A^\prime$ since the use of $A$ is bounded.  We can then determine if $n \in A^b$.
\end{proof}

While the bounded jump is not very interesting from the perspective of the Turing degrees, we hope to show it follows our intuition for a jump on the bounded Turing degrees.

We start by showing that the bounded jump is strictly increasing.  The proof is a diagonalization argument using the Recursion Theorem.

\begin{theorem} Let $A$ be any set.  Then $A^b \not \leq_{bT} A$. \label{strinc} \end{theorem}
\begin{proof} Suppose not.  Let $\Gamma$ and $g$ witness $A^b \leq_{bT} A$.  We define a computable function $f$ by
$$\Phi_{f(e)}^C (x) = \begin{cases} 0 & \; x \not= e \text{ or } (x = e \text{ and } \Gamma^C (e) = 0)\\ \Phi_e^C (e) + 1 & \; x = e \text{ and } \Gamma^C (e) = 1\\ \dive & \; x = e \text{ and } \Gamma^C (e) \dive \end{cases}$$

By the Recursion Theorem, let $M$ be an infinite computable set such that for all $m \in M$ we have $\Phi^C_m = \Phi^C_{f(m)}$.  Let $k$ be such that $g = \varphi_k$ and pick $m \in M$ such that $m > k$.  We note $\Gamma^A$ is total so $\Phi^A_{f(m)}$ is total and thus $\Phi^A_m$ is total.

Suppose $\Gamma^A(m) = 1$.  Then $\Phi^A_m (m) = \Phi^A_{f(m)} (m) = \Phi^A_m (m) + 1$ for a contradiction.

Hence $\Gamma^A(m) = 0$.  Thus $m \notin A^b$.  So for all $i \leq m$ with $\varphi_i (m) \conv$ we have $\Phi_m^{A \smrstrd \varphi_i (m)} (m) \dive$.  In particular, since $k < m$ and $\varphi_k(m) = g(m) \conv$ we have $\Phi_m^{A \smrstrd g(m)} (m) \dive$.  Thus $\Phi_{f(m)}^{A \smrstrd g(m)} (m) \dive$ so $\Gamma^{A \smrstrd g(m)} (m) \dive$.  This contradicts our choice of $\Gamma$ and $g$.

We conclude $A^b \not \leq_{bT} A$.
\end{proof}

We next show the bounded jump is order-preserving on the bounded Turing degrees.  The proof is a careful application of the $s$-$m$-$n$ Theorem.

\begin{theorem} Let $A$ and $B$ be sets with $A \leq_{bT} B$.  Then $A^{b_0} \leq_1 B^{b_0}$. \end{theorem}
\begin{proof} Let $\Psi$ and $f$ witness $A \leq_{bT} B$\@.  By the $s$-$m$-$n$ Theorem, let $h$ be a strictly increasing computable function such that $\varphi_{h(i)} (x) = f(\varphi_i (x))$.  Since $f$ is total, $\varphi_{h(i)} (x) \conv \iff \varphi_i (x) \conv$.

We define a computable, injective function $g$ by
$$\Phi^C_{g(\langle e, k, j \rangle)}(x) = \begin{cases} \Phi_e^{(\Psi^{C \smrstrd \varphi_{h(i)} (j)})\smrstrd \varphi_i (j)} (j) & \; k = h(i) \text{ for some $i$ and } \varphi_i (j) \conv \\ \dive & \; \text{else} \end{cases}$$

We now note:

\begin{align*}
\langle e, i, j \rangle \in A^{b_0}
\iff \ & \varphi_i (j) \conv \mbox{ and } \Phi_e^{A \smrstrd \varphi_i (j)} (j) \conv \\
\iff \ & \varphi_{h(i)} (j) \conv \mbox{ and } \Phi_e^{A \smrstrd \varphi_i (j)} (j) \conv \\
\iff\ & \varphi_{h(i)} (j) \conv \mbox{ and }  \Phi_e^{(\Psi^{B \smrstrd \varphi_{h(i)} (j)})
    \smrstrd \varphi_i (j)} (j) \conv \\
\iff \ & \varphi_{h(i)} (j) \conv \mbox{ and } \Phi_{g(\langle e, h(i),
j \rangle)}^{B \smrstrd \varphi_{h(i)}(j)} (j) \conv \\
\iff \ & \langle g(\langle e, h(i), j \rangle), h(i), j
\rangle \in B^{b_0}
\end{align*}

Therefore $A^{b_0} \leq_1 B^{b_0}$.
\end{proof}
Since for any set $X$ we have $X^{b_0} \equiv_{tt} X^b$, we
immediately obtain the following corollary.
\begin{cor} Let $A$ and $B$ be sets with $A \leq_{bT} B$.  Then $A^b \leq_{tt} B^b$. \label{ordprev} \end{cor}

We would also like to show that $A^b$ is not equivalent to $A \oplus \es^\prime$ for
the bounded Turing degrees.  We noted earlier that this holds
on the cone above $\es^\prime$.  We can also demonstrate this
holds elsewhere.  We recall two notions of sets being
``ordinary'' (see Nies \cite{Nies} for more information on
randomness).

\begin{definition} $X$ is $n$-generic if for every $\Sigma_n$ set $S \subseteq \strings$ either $X$ meets $S$ (i.e.\ there is an initial segment of $X$ in $S$) or there is an $l \num$ such that every string $\sigma$ extending $X \rstrd l$ is such that $\sigma \notin S$.  \end{definition}
\begin{definition} $X$ is $n$-random if for every uniformly $\Sigma_n$ family of sets
$\langle U_i \subseteq \strings \setsep i \num \rangle$ such
that $\mu (U_i) \leq 2^{-i}$ for all $i$, there exists an $l$
such that $X$ does not meet $U_l$.  \end{definition}

We show that if $A$ is 3-generic or 4-random then $A^b \not
\leq_{bT} A \oplus \es^\prime$.  In the proof, we will assume
towards a contradiction that $\Psi$ and $f$ witness $A^{b_0}
\leq_{bT} A \oplus \es^\prime$.  We will then use the Recursion
Theorem to find a computable set whose elements $n$ are such
that $f(n)+1 \in A  \iff n \in A^{b_0}$. Since $A \rstrd f(n)$
computes $A^{b_0} (n)$, we have $A (f(n) + 1)$ predicted by $A
\rstrd f(n)$.  This regularity property can then be used to
show $A$ is not 3-generic or 4-random, for a contradiction.

\begin{theorem}\label{thm_3generic} Let $A$ be 3-generic.  Then $A^b \not \leq_{bT} A \oplus \es^\prime$. \end{theorem}
\begin{proof} Suppose not.  Then $A^{b_0} \leq_{bT} A \oplus \es^\prime$.  Let $\Psi$ and $f$ be such that \\$\Psi^{A \smrstrd f(n) \oplus \es^\prime \smrstrd f(n)} (n) = A^{b_0} (n)$ for all $n$.

For any $e$, let $\varphi_{g_e(i)}(j) = f(\langle e, i, j
\rangle) + 1$. By the Recursion Theorem, let $Z_e$ be an
infinite computable set such that for all $m \in Z_e$ we have
$\varphi_{g_e(m)} = \varphi_m$.  We choose $g_e$ and $Z_e$ such that they are uniformly computable.  Let $t$ be a computable
function such that for all $e$ we have $t(e) \in Z_e$. Then for
all $e, j \in \omega$,
$\varphi_{t(e)}(j)=\varphi_{g_e(t(e))}(j) = f( \la e, t(e), j
\ra) +1$. In particular, $\varphi_{t(e)}$ is total.

We define a computable function $h$ by
$$\Phi_{h(e)}^C (j) = \begin{cases} 1 & \; C(\varphi_{t(e)} (j)) = 1\\ \dive & \; \text{else} \end{cases}$$

By the Recursion Theorem, let $H$ be an infinite computable set
such that for all $m \in H$ we have $\Phi^C_{h(m)} = \Phi^C_m$.

We then have for every $n \in H$ that
$$\Phi_n^C (j) = \begin{cases} 1 & \; C(f(\langle n, t(n), j \rangle + 1)) = 1\\ \dive & \; \text{else} \end{cases}$$

Hence for every $n \in H$ we have
\begin{align*}
\langle n, t(n), j \rangle \in A^{b_0}
\iff \ & \varphi_{t(n)}(j) \dn \wedge \Phi_n^{A \Eres \varphi_{t(n)}(j)}(j) \dn \\
\iff \ & \Phi_n^{A \Eres f(\la n, t(n), j \ra) +1}(j) \dn \\
\iff \ & A(f(\langle n, t(n), j \rangle) + 1) = 1
\end{align*}

Thus for $n \in H$ we have $$\Psi^{A \smrstrd f(\langle n,
t(n), j \rangle) \oplus \es^\prime \smrstrd f(\langle n, t(n),
j \rangle)} (\langle n, t(n), j \rangle) = 1 \iff A(f(\langle
n, t(n), j \rangle) + 1) = 1$$

We define a set $S$ by
$$S = \{ \sigma \in \strings \setsep \exists j \, \exists n \in H [\mbox{length}(\sigma)>f(\la n, t(n), j \ra) \mbox{ and }$$
$$\Psi^{\sigma \smrstrd f(\langle n, t(n), j \rangle) \oplus \es^\prime \smrstrd f(\langle n, t(n), j \rangle)} (\langle n, t(n), j \rangle) \conv \not= \sigma(f(\langle n, t(n), j \rangle)+1) \mbox{ or diverges}] \}$$

We note $S$ is $\Sigma_2 (\es^\prime)$ so $S$ is $\Sigma_3$.  Since $A$ is 3-generic and $A$ does not meet $S$, there is an $m$ such that for all $\tau$ extending $A \rstrd m$ we have $\tau \notin S$.  However, any string $\tau$ can be extended to one in $S$ by picking a value for $\tau(f(\langle n, t(n), j \rangle) + 1)$ that disagrees with the prediction of $\Psi$ (if it converges) for some sufficiently large $n \in H$ and $j$.  This is a contradiction so we conclude $A^b \not \leq_{bT} A \oplus \es^\prime$.
\end{proof}

A similar proof can be used to show that if $A$ is 4-random then $A^b \not \leq_{bT} A \oplus \es^\prime$.  Hence, for the bounded Turing degrees, the class of sets where $A^b$ is equivalent to $A \oplus \es^\prime$ has measure zero.

\begin{cor} Let $A$ be 4-random.  Then $A^b \not \leq_{bT} A \oplus \es^\prime$. \end{cor}
\begin{proof} Suppose not.  Then $A^{b_0} \leq_{bT} A \oplus \es^\prime$.  Let $\Psi$ and $f$ be such that \\$\Psi^{A \smrstrd f(n) \oplus \es^\prime \smrstrd f(n)} (n) = A^{b_0} (n)$ for all $n$.

Let $t$, $H$, and $S$ be as in the proof Theorem
\ref{thm_3generic}.  Since $f$, $t$, and $H$ are computable, we
can find a computable, strictly increasing function $l$ such
that for all $m \num$ we have $l(m) = f(\la n, t(n), j \ra)$
for some $j$ and some $n \in H$.

For each $i \num$ let $U_i = \{ \sigma \setsep \sigma \notin S$ and length$(\sigma) = l(i)+1 \}$.  We note that the $U_i$ are uniformly $\Pi_3$ since $S$ is $\Sigma_3$ and $l$ is computable.  Since $A$ does not meet $S$, we know that $A$ meets every $U_i$.

We note from the definition of $S$ that if $\tau$ is any string of length $l(m)$ for some $m$, then at least one of $\tau \concat 0$ and $\tau \concat 1$ is in $S$.  We also note $S$ is closed under extensions so if $\sigma \notin U_i$, length$(\sigma) \geq l(i)$, and $\rho$ extends $\sigma$ then $\rho \notin U_j$ for any $j \geq i$.  Hence $\mu(U_i) \leq 2^{-i}$.  We conclude that $A$ is not 4-random, for a contradiction.  Thus $A^b \not \leq_{bT} A \oplus \es^\prime$.
\end{proof}

\section{Ershov Hierarchy}

The iterates of the jump correspond to completeness in the arithmetic hierarchy; the $n$-th jump is $\Sigma_n$ complete.  We will show that the iterates of the bounded jump correspond to completeness in the Ershov hierarchy.

Fix a canonical, computable coding of the ordinals less than $\omega^\omega$.  Since we do not use ordinals above $\omega^\omega$ in this paper, the details of the coding are not significant.  We say a function on an ordinal $\alpha$ is (partial) computable if the corresponding function on codes for the ordinal $\alpha$ is (partial) computable.

For $\alpha \geq \omega$, we say that a set $A$ is $\alpha$-c.e.\ if there is a partial computable $\psi:\omega \times \alpha \to \{0,1\}$ such that for every $n \in \omega$, there exists a $\beta < \alpha$ where $\psi (n, \beta) \conv$ and $A(n) = \psi (n, \gamma)$ where $\gamma$ is least such that $\psi (n, \gamma) \conv$ \cite{REA}.

It is a well known result that $X \leq_{bT} \es' \iff X
\leq_{tt} \es^\prime \iff X$ is $\omega$-c.e.\ \cite{Nies}.
Using the bounded jump, this is $X \leq_{bT} \es^b \iff X \leq_{tt} \es^b \iff X$ is
$\omega$-c.e.  We wish to extend this observation to higher
powers of $\omega$.  In fact, we are able to establish a
slightly stronger result.

\begin{theorem} For any set $X$ and $n \geq 2$ we have $X \leq_{bT} \es^{nb} \iff X$ is $\omega^n$-c.e.\ $\iff X \leq_1 \es^{nb}$.  \label{ErshovThm} \end{theorem}

A set $A$ is a $tt$-cylinder if for all $X$ we have $X \leq_{tt} A \Rightarrow X \leq_1 A$ \cite{OdifreddiCRT}.

\begin{cor} For all $n \geq 2$ we have $\es^{nb}$ is a $tt$-cylinder.  \end{cor}

The theorem follows from the lemmas below.  We first introduce some notation.  Let $+_c$ denote commutative addition of ordinals (term-wise sum of coefficients of ordinals in Cantor normal form) \cite{AshKnight}.  We will use two properties of commutative addition.  First, given $\alpha_1 \ldots \alpha_n$ and $\beta_1 \ldots \beta_n$ such that $\beta_i \leq \alpha_i$ for all $i \leq n$ and $\beta_j < \alpha_j$ for some $j \leq n$ then $\beta_1 +_c \beta_2 +_c \ldots +_c \beta_n < \alpha_1 +_c \alpha_2 +_c \ldots +_c \alpha_n$.  Also, if for some $\gamma$ we have $\alpha_i < \omega^\gamma$ for all $i$, then $\alpha_1 +_c \alpha_2 +_c \ldots +_c \alpha_n < \omega^\gamma$.

We start by proving that being $\omega^k$-c.e.\ is closed
downward in the bounded Turing degrees.  For the proof, we
suppose $\Phi$ and $f$ witness $A \leq_{bT} B$ and $\psi$
witnesses $B$ is $\omega^k$-c.e.  We will then build $\chi$ to
witness $A$ is $\omega^k$-c.e. In order to estimate $A(n)$, we
will estimate $B\Eres f(n)$ using $\psi(i, \alpha_i)$ for $i
\leq f(n)$, and record the output of $\Phi$ on this estimate at
$\chi (n, \alpha_1 +_c \ldots +_c \alpha_{f(n)})$.

\begin{lemma} Let $k > 0$ and let $A$ and $B$ be sets such that $A \leq_{bT} B$ and $B$ is $\omega^k$-c.e.  Then $A$ is $\omega^k$-c.e. \end{lemma}
\begin{proof}
Let $\Phi$ and $f$ witness $A \leq_{bT} B$ and let $\psi$
witness $B$ is $\omega^k$-c.e. We will define a function $\chi$
to witness that $A$ is $\omega^k$-\ce by stages as follows.
Fix $n$ (we simultaneously follow the same procedure for each
$n$).

At each stage $s$, for $i \leq f(n)$, let $\alpha_i^s$ be the
least ordinal such that $\psi_s(i, \alpha_i^s) \dn$, if it
exists. Define a string $\sigma_s(\alpha_0^s, \ldots,
\alpha_{f(n)}^s)$ of length $f(n)$ by letting $\sigma_s (i) =
\psi_s (i, \alpha_i^s)$.

Let $s_0$ be the least stage where $\alpha_i^{s_0}$ are defined
for all $i \leq f(n)$. Set $$\chi (n, \alpha_0^{s_0} +_c \ldots
+_c \alpha_{f(n)}^{s_0})) = \Phi^{\sigma_{s_0}(\alpha_0^{s_0},
\ldots, \alpha_{f(n)}^{s_0})} (n).$$ Note that $\alpha_0^{s_0}
+_c \ldots +_c \alpha_{f(n)}^{s_0} < \omega^k$.

At stage $s+1 > s_0$, if $\alpha_i^{s+1} < \alpha_i^s$ for some
$i \leq f(n)$ then define
$$\chi (n, \alpha_0^{s+1} +_c \ldots
+_c \alpha_{f(n)}^{s+1})) = \Phi^{\sigma_{s+1}(\alpha_0^{s+1},
\ldots, \alpha_{f(n)}^{s+1})} (n).$$ This is possible since
$\alpha_0^{s+1} +_c \ldots +_c \alpha_{f(n)}^{s+1} <
\alpha_0^{s} +_c \ldots +_c \alpha_{f(n)}^{s}$.

It is clear that $\chi$ is partial recursive.  Let $n$ be arbitrary, and for $i \leq f(n)$ let $\beta_i$ be least such that $\psi (i, \beta_i) \conv$.  Let $\gamma = \beta_0 +_c \ldots +_c \beta_{f(n)}$.  Then $\gamma$ is least such that $\chi (n, \gamma) \conv$ and $\chi (n, \gamma) = A(n)$.  Thus $\chi$ witnesses $A$ is $\omega^k$-c.e.
\end{proof}

We next prove that if $A$ is $\omega^k$-c.e.\ then $A^b$ is $\omega^{k+1}$-c.e.  Combined with the previous lemma this will give us that $X \leq_{bT} \es^{nb} \Rightarrow X$ is $\omega^n$-c.e.

For the proof, we will let $\psi$ witness that $A$ is $\omega^k$-c.e.\ and will define $\chi$ to witness that $A^b$ is $\omega^{k+1}$-c.e.  We will start with $\chi(n, \omega^k \cdot n) = 0$ and each time we witness a new, longer $\varphi_i(n) \conv$ for some $i \leq n$ we will move down to a new $\omega^k$ level.  At a fixed level, we will record estimates of $A^b$ based on estimates of $A \rstrd \varphi_i(n)$ in a manner similar to the previous lemma.

\begin{lemma} Let $k > 0$ and let $A$ be a set such that $A$ is $\omega^k$-c.e.  Then $A^b$ is $\omega^{k+1}$-c.e.  \label{jumpErshov} \end{lemma}
\begin{proof} Let $\psi$ witness that $A$ is $\omega^k$-c.e.  We will define a function $\chi$ to witness that $A^b$ is $\omega^{k+1}$-c.e.  Fix $n$ (we simultaneously follow the same procedure for each $n$).

For an ordinal $\beta$, let $u(\beta)$ be the coefficient of
the units digit of $\beta$ in Cantor normal form.  We again let
$\alpha_i^s$ be the least ordinal such that $\psi_s(i,
\alpha_i^s) \dn$, if it exists, and define a string
$\sigma_s(\alpha_0^s, \ldots, \alpha_{m}^s)$ of length $m$ by
letting $\sigma_s (i) = \psi_s (i, \alpha_i^s)$.  Indeed, we
will assume wlog that $\alpha_i^s$ is defined at each stage $s$
by running the computation $\psi$ for longer than $s$ steps if
necessary.

Let $r(l, \alpha_0^s, \ldots, \alpha_m^s) = \omega^k
\cdot l +_c \alpha_0^{s} +_c \ldots +_c \alpha_{m}^{s} +_c
u(\alpha_0^{s} +_c \ldots +_c \alpha_{m}^{s})$.  Note that if all
$\alpha_i^s < \omega^k$, then $r(l, \alpha_0, \ldots, \alpha_m) <
\omega^{k+1}$. Note also that if $\beta_i \leq \alpha_i$ for
all $i \leq m$ and $l' \leq l$, and if one of the inequalities
is strict, then $r(l', \beta_0, \ldots, \beta_m) +2 < r(l,
\alpha_0, \ldots, \alpha_m) +1$.

We let $\chi (n, \omega^k \cdot n) = 0$ and set bookkeeping
variables $l_0 = n$ and $m_0 = 0$.  Every time we see
$\varphi_i (n) \conv > m$ for some $i \leq n$ we will decrease $l$
by one and let $m = \varphi_i (n)$. We note this can
happen at most $n$ many times.

At stage $s+1$: If $\varphi_{i, s+1} (n) \conv > m_s$ for some
$i \leq n$, define $l_{s+1}=l_s -1$ and let $m_{s+1} =
\varphi_i (n)$. Otherwise, let $l_{s+1}=l_s$ and $m_{s+1} =
m_s$.

If $l_{s+1}< l_s$ or if $\alpha_i^{s+1} < \alpha_i^s$ for some
$i \leq m_{s+1}$, then let $\chi (n, r(l, \alpha_0, \ldots,
\alpha_m) + 2)[s+1] = 0$.

If $\Phi_n^{\sigma(l, \alpha_0, \ldots, \alpha_m)} (n) \conv
[s+1]$, then set $\chi (n, r(l, \alpha_0 \ldots \alpha_m) + 1)
[s+1] = 1$.

This completes the construction.

We note $\chi$ is partial recursive.  Let $n$ be arbitrary and
let $m$ be the largest value of $\varphi_i(n)$ for $i \leq n$
such that $\varphi_i(n) \conv$.  For $j \leq m$, let $\beta_j$
be least such that $\psi(j, \beta_j) \conv$.  Let $l$ be least
such that for some $\delta < \omega^k$ we have $\chi (n,
\omega^k \cdot l + \delta) \conv$.  Then $\chi (n, r(l, \beta_0
\ldots \beta_m) + 2) = 0$ and $\chi (n, r(l, \beta_0 \ldots
\beta_m) + 1) \conv = 1$ iff $n \in A^b$.  For all $\gamma \leq
r(l, \beta_0 \ldots \beta_m)$ we have $\chi (n, \gamma) \dive$.
Therefore $\chi$ witnesses $A^b$ is $\omega^{k+1}$-c.e.
\end{proof}

We note that the proofs for the above lemmas hold for any
ordinal $\omega^\gamma$ such that $0 < \omega^\gamma <
\omega_1^{CK}$.

To complete the proof of the theorem, we wish to show that if $A$ is $\omega^k$-c.e.\ then $A \leq_1 \es^{kb}$.  We start by proving the statement for $k = 2$.

For the proof, suppose $\psi$ witnesses that $A$ is $\omega^2$-c.e.  Let $n \num$ and let $m$ be first such that we see $\psi (n, \omega \cdot m + j)\conv$ for some $j$.  To determine if $n \in A$, we need to know enough of $\es^b$ to answer the $\Sigma_1$ questions $\exists j \,[\psi(n, \omega \cdot i + j)\conv]$ for each $i < m$.  In each case, if the answer is yes, first witnessed by $\omega \cdot i + k$, we then need to know if $\psi(n, \omega \cdot i + j)\conv$ for all $j < k$.  There is no computable bound which can be determined in advance stating how much of $\es^b$ is needed to answer all of these questions.  However, we can in advance bound the indices of the computable functions needed to determine how much of $\es^b$ will be used.  Hence we can bound the amount of $\es^{2b}$ required to have enough access to $\es^b$ to answer these questions.

\begin{lemma} Let $A$ be a set such that $A$ is $\omega^2$-c.e.  Then $A \leq_1 \es^{2b}$. \label{Erbase} \end{lemma}
\begin{proof} Let $\psi$ witness that $A$ is $\omega^2$-c.e.  We will define several functions, ending in a computable $f$ such that $n \in A \iff f(n) \in \es^{2b}$.

Let $g$ be a computable function such that $g(n) = i$ where the first time we observe $\psi(n,\alpha) \conv$ is $\alpha = \omega \cdot i + j$ for some $j$.  Let $q(i,n)$ be the first $m$ observed such that $\psi(n, \omega \cdot i + m) \conv$.  The function $q$ is partial computable since it may be there is no such $m$ for the given $i$.

Let $\tilde{h}(i,x,n)$ denote the spot of $\es^b$ which answers the question $\exists m \leq x$ \\ $[\psi(n,\omega \cdot i + m)\conv]$.  Let $\tilde{r}(n,i)$ denote the spot of $\es^b$ which answers the question $\exists m [\psi(n,\omega \cdot i + m)\conv]$.  We then let $h(i,n) = \max \{ \tilde{h}(i,x,n) \setsep x \leq q(i,n) \}$ and $r(n) = \max \{ \tilde{r}(n,x) \setsep x \leq g(n) \}$.  The functions $\tilde{h}, \tilde{r},$ and $r$ are computable and $h$ is partial computable, converging wherever $q$ does.

Let $p(n)$ be the least $i$ such that for some $m$ we have $\psi(n,\omega \cdot i + m)\conv$.  We can compute $p(n)$ from $\es^b \rstrd r(n)$.  We note that $h(p(n),n)$ exists and we can determine if $n \in A$ from $\es^b \rstrd \max \{r(n), h(p(n),n) \}$.

Let $v$ be a computable function defined by $\varphi_{v(i,n)} (y) = h(i,n) + r(n)$ ($y$ is a dummy variable).  Let $u(n) = \max \{ v(i,n) \setsep i \leq g(n) \}$.  The function $u$ is computable and if we let $j = v(p(n),n)$ then $j \leq u(n)$, the function $\varphi_j (y) \conv$, and $\es^b \rstrd \varphi_j (y)$ suffices to determine if $n \in A$ (for any $y$).

We now define $f(n) > u(n)$ to be such that (for any $y$), $\Phi^{\es^b}_{f(n)}(y)$ runs the calculation to determine if $n \in A$, and converges iff $n \in A$.  Explicitly, we define $f(n) > u(n)$ such that $\Phi^C_{f(n)} (y)$ is the partial computable function determined by the following steps.  First, we let $x \leq g(n)$ be least such that $C(\tilde{r}(n,x))=1$.  Next, we let $t$ be first such that we observe $\psi(n,\omega \cdot x + t)\conv$.  We then let $z \leq t$ be least such that $C(\tilde{h}(x,z,n))=1$.  Finally, we say $\Phi^C_{f(n)}(y) \conv$ if $\psi(n,\omega \cdot x + z) = 1$ and $\Phi^C_{f(n)}(y) \dive$ if $\psi(n,\omega \cdot x + z) = 0$ (or if any of the above steps can't be completed).

We note that $f$ is computable, and if $C$ is a sufficiently long initial segment of $\es^b$ then $n \in A$ iff $\Phi^C_{f(n)}(y) \conv$.  Recall that for any $n$, there exists $j \leq u(n) < f(n)$ such that $\varphi_j (y) \conv$ and $\es^b \rstrd \varphi_j (y)$ suffices to run the calculations to determine if $n \in A$.

We observe $f(n) \in \es^{2b} \iff \exists i \leq f(n)
[\varphi_i(f(n)) \conv \ \wedge\ \Phi_{f(n)}^{\es^b \smrstrd
\varphi_i(f(n))} (f(n)) \conv]\\ \iff n \in A$.  Hence $f$
witnesses $A \leq_1 \es^{2b}$.
\end{proof}

We use a similar method to prove the statement for all $k$.

\begin{lemma} Let $k > 1$ and let $A$ be a set such that $A$ is $\omega^k$-c.e.  Then $A \leq_1 \es^{kb}$. \end{lemma}
\begin{proof} We prove the statement by induction on $k$.  The base case ($k=2$) is given by Lemma \ref{Erbase}.  For the inductive case, we assume the statement holds for $k$ and wish to show it holds for $k+1$.  We note for the procedure given in Lemma \ref{Erbase} that an index for $f$ can be computed uniformly from an index for $\psi$.

The proof for the inductive case proceeds along the same lines as the proof for the base case.  Let $\psi$ witness that $A$ is $\omega^{k+1}$-c.e.  Let $g$ be a computable function such that $g(n) = i$ where the first time we observe $\psi(n,\alpha) \conv$ is $\omega^k \cdot i + \alpha$ for some $\alpha < \omega^k$.  Let $p(n)$ be the least $i$ such that for some $\alpha$ we have $\psi(n,\omega^k \cdot i + \alpha)\conv$.

Let $\chi_i (n, \alpha) = \psi (n, \omega^k \cdot i + \alpha)$ for all $\alpha < \omega^k$.  We define a partial computable sequence of functions $e_i (n)$ as follows.  To compute $e_i (n)$ we first search for any $\alpha$ such that $\chi_i (n, \alpha) \conv$.  If there is none then we must have $e_i (n) \dive$.  If the search halts then let
$$\tilde{\chi} (m, \alpha) = \begin{cases} \chi_i(m, \alpha) & \; m=n\\0 & \; \text{else} \end{cases}$$
Let $B$ be such that $\tilde{\chi}$ witnesses $B$ is $\omega^k$-c.e.\ and let $\tilde{f}$ be given by applying the induction hypothesis to $B$.  We then let $e_i (n) = \tilde{f} (n)$.

Let $v$ be a computable function defined by $\varphi_{v(i,n)} (y) = e_i (n)$ ($y$ is a dummy variable).  Let $u(n) = \max \{v(i,n) \setsep i \leq g(n) \}$.  The function $u$ is computable and if we let $j = v(p(n),n)$ then $j \leq u(n)$, the function $\varphi_j (y) \conv$, and $\es^{kb} \rstrd \varphi_j (y)$ suffices to determine $e_{p(n)}(n)$ and hence if $n \in A$ (for any $y$).

We define $f(n) > u(n)$ to be such that (for any $y$), $\Phi^{\es^{kb}}_{f(n)}(y)$ calculates if $n \in A$, and converges iff $n \in A$.  Explicitly, we define $f(n) > u(n)$ such that $\Phi^C_{f(n)} (y)$ is the partial computable function determined by the following steps.  Let $l$ be such that $\Phi^{\es^{kb}}_l (m) = p(m)$.  We then have $\Phi_{f(n)}^C (y)$ converge iff $\Phi^C_l (n)$ converges and $e_{\Phi^C_l (n)}(n) \in C$.  As in the proof of Lemma \ref{Erbase}, $f$ witnesses  $A \leq_1 \es^{(k+1)b}$, completing the induction.

Therefore for all $k \geq 2$ we have $A$ is $\omega^k$-c.e.\ $\Rightarrow A \leq_1 \es^{kb}$.
\end{proof}

We proved earlier that for any set $A$ we had $A^{b_0} \leq_1 A^b$ and $A^b \leq_{tt} A^{b_0}$.  However, we can use the results above to show that $A^b$ and $A^{b_0}$ are not always 1-equivalent.

\begin{remark} $\es^{2b} \not\leq_1 (\es^b)^{b_0}$.  \label{notoneequiv} \end{remark}
\begin{proof} Suppose $\es^{2b} \leq_1 (\es^b)^{b_0}$.  Let $A$ be a properly $\omega^2$-c.e.\ set.  By Theorem \ref{ErshovThm}, we have $A \leq_1 \es^{2b}$ so $A \leq_1 (\es^b)^{b_0}$.

Using an argument similar to the proof of Lemma
\ref{jumpErshov}, we can show that $(\es^b)^{b_0}$ is
$(\omega+1)$-c.e. Indeed, while $\varphi_i(j)\up$ we believe
$\la e, i, j \ra \nin (\es^b)^{b_0}$, and once $\varphi_i(j)
\dn$, since $\es^b$ is \ce, we can approximate
$(\es^b)^{b_0}(\la e, i, j \ra)$ with at most $2(\varphi_i(j))$
many changes. It is not hard to see that for any sets $B$ and $C$
and any ordinal $\alpha$, if $B \leq_1 C$ and $C$ is
$\alpha$-c.e., then $B$ is $\alpha$-c.e. Hence $A$ is
$(\omega+1)$-c.e., contradicting $A$ being properly
$\omega^2$-c.e.  We conclude $\es^{2b} \not\leq_1
(\es^b)^{b_0}$.
\end{proof}

\section{Inversions}

We examine what type of inverses exist for the bounded jump.  Anderson \cite{andtt} proved that strong jump inversion holds for the truth-table degrees.  For any set $X \geq_{tt} \es^\prime$ there is a set $Y$ such that $X \equiv_{tt} Y^\prime \equiv_{tt} Y \oplus \es^\prime$.  It follows as a corollary that strong bounded jump inversion holds for the truth-table degrees.

\begin{cor} Let $X \geq_{tt} \es^b$.  Then there exists $Y$ such that $Y^b \equiv_{tt} X \equiv_{tt} Y \oplus \es^b$.  \label{bounded sji} \end{cor}
\begin{proof} Let $X$ be given and let $Y$ be given by strong jump inversion for the truth-table degrees.  Then $Y ^\prime \equiv_{tt} X \equiv_{tt} Y \oplus \es^b$ and from section \ref{Properties} we have $Y \oplus \es^b \leq_{tt} Y^b$ and $Y^b \leq_{tt} Y^\prime$.  We conclude $Y^b \equiv_{tt} X \equiv_{tt} Y \oplus \es^b$.
\end{proof}

A close examination of the proof in \cite{andtt} reveals that an equivalent statement also holds for the bounded Turing degrees.  For any set $X \geq_{bT} \es^\prime$ there is a set $Y$ such that $X \equiv_{bT} Y^\prime \equiv_{bT} Y \oplus \es^\prime$.  If we apply the proof of Corollary \ref{bounded sji}, we get that for any set $X \geq_{bT} \es^b$ there is a set $Y$ such that $X \equiv_{bT} Y^b \equiv_{bT} Y \oplus \es^b$.

As noted earlier, Shoenfield jump inversion \cite{ShoenfieldInv} holds for the Turing degrees with the Turing jump, for every $\Sigma_2$ set $X \geq_T \es^\prime$ there is a $Y \leq_T \es^\prime$ such that $Y^\prime \equiv_T X$.  Csima, Downey, and Ng \cite{CsimaDowneyNg} showed that it does not hold for the bounded Turing degrees with the Turing jump.

We prove that Shoenfield jump inversion holds for the bounded Turing degrees with the bounded jump.  In this example, the behavior of the bounded jump on the bounded Turing degrees more closely resembles the behavior of the Turing jump on the Turing degrees.

\begin{theorem} Let $B$ be such that $\es^b \leq_{bT} B \leq_{bT} \es^{2b}$.
Then there is a set $A \leq_{bT} \es^b$ such that $A^b
\equiv_{bT} B$. \label{SchThm} \end{theorem}
\begin{proof}

Suppose $\es^b \btl B \btl \es^{2b}$. Let $\psi$ witness that
$B$ is $\omega^2$-c.e. We build an $\omega$-\ce set $A$ (so $A
\btl \es^b$) such that $A^b \equiv_{bT} B$.

We will define $A$ using a stage by stage construction. We will ensure that $A$ is
$\omega$-\ce via the function $f(x)= x+1$. Before we start, we define a computable
function $g$. We will have $g$ witness that $B \leq_1 A^b$.

For each $n \in \omega$, let $i_n$ be the first $i$ that we find such that $\psi(n, \omega \cdot i + j)\conv$, 
for some $j$.  The definition of $\psi$ guarantees such an $i$ exists, so the $i_n$ are uniformly computable.  We 
define an approximation $B_s$ for $B$ similarly.  Fix $n$ and let $t$ be least such that $\psi_t (n, \alpha) \conv$, 
for some $\alpha$.  Given $s$, let $\tilde{s} = \max(t,s)$ and let $B_s (n) = \psi (n, \beta)$, where $\beta$ is least 
such that $\psi_{\tilde{s}} (n, \beta) \conv$.

We define a computable function $h$ to help define $g$. Let
$g(-1) = -1$. Let $h(n)= \Sigma^{n-1}_{k=0} h(k) + \Sigma^{g(n-1)}_{k=1}(\frac{k^2 -k}{2}) + i_{n}$.
Let $g(n)$ be such that between $g(n-1)$ and $g(n)$ there are $h(n)$-many
partial computable functions $\varphi_{k(n,0)},\ \ldots\ ,$\\$\varphi_{k(n, h(n) - 1)}$
that we control by the recursion
theorem, and such that we control $\Phi_{g(n)}$ by the
recursion theorem.  The formal definitions of $g$, $h$, and $k$ are 
given at the end of the proof in Lemma \ref{ghlem}.

We will make use of markers labeled $x_n^i$ with $i \leq i_n$,
called \emph{$n$-markers}, which will move stage by stage, but
reach a limit. At some stage $s$, we might say that a marker
$x_n^i$ becomes \emph{defined}. The marker then maintains its
value, unless it becomes \emph{undefined} at a later stage. If
it at an even later stage becomes redefined, then it will have
a new, larger, value. At any moment, there will be at most one
$n$-marker defined for each $n$. There will be a computable
bound on the total number of times all $n$-markers will be
defined/redefined, namely $h(n)$.

In each stage of the construction we will make numerous changes
to the approximation of the set $A$. To ease notation, when we
write ``$A$" in the construction, we actually mean the most
current approximation of $A$ at that moment of the
construction. By ``$A_s$" we mean the approximation $A$ at the
\emph{end} of stage $s$.  Without loss of generality, we assume
that if $\psi_s (n,\omega \cdot i + j)\conv$ then the stage
$s > j + 1$.

\emph{Stage $s$:}

\emph{Step $1$:} If some $\varphi_{e,s}(x)\conv$ for the first
time at stage $s$, with $e \leq x \leq g(k)$, then for all $m
>k$, extract all $x^l_m$ from $A$, and declare them
undefined.

\emph{Step $2$:} Let $n \leq s$ be least such that $x^i_n$ is
defined, but $A(x^i_n) \not= B_s(n)$, or such that $\psi(n,
\alpha) \conv$ for some $\alpha$ but there is no marker defined
for $n$. Let $\omega \cdot i + j$ be least such that $\psi_s
(n, \omega \cdot i + j ) \conv$.

(a) If $x_n^i$ is undefined, then we perform the following steps.
Define $x_n^i = s$.  Extract all $x^l_m$ with $m >n$ and all
$x^k_n$ with $k >i$ from $A$, and declare them undefined. Define
$\varphi_{k(n,r)}(g(n))\conv = x^i_n$ for some $r$, and declare
$\Phi_{g(n)}^{\sigma \concat 1}(g(n)) \conv$ for every string
$\sigma$ of length $x^i_n - 1$. Note that by our assumption,
$j+1 < s =x_n^i$. There will always be some $r$ with
$\varphi_{k(n,r)}(g(n))[s-1] \dive$ by our careful counting of $h(n)$.

(b) If needed, change $A_s (x_n^i)$ to ensure $x_n^i \in A_s$ iff
$n \in B_s$ (so that $g(n) \in A^b_s$ iff $n \in B_s$).

This completes the construction.

\begin{lemma}
$A$ is $\omega$-\ce
\end{lemma}
\begin{proof}
If at stage $s$ we did not set $s = x^i_n$ for any $n$, then
$s$ was never enumerated into $A$. If at stage $s$ we set $s=
x^i_n$, then $x^i_n$ is enumerated into $A$, and can be
removed/enumerated into $A$ at most $j$-many more times by Step
$1$ of the construction (where $j$ is least such that
$\psi_s(n, \omega \cdot i +j) \conv$). By convention $j \leq s$,
so certainly $s$ is enumerated/removed from $A$ at most
$s+1$-many times.
\end{proof}
\begin{lemma}
For each $n$ and each $i \leq i_n$, $x^i_n = \lim x^i_n[s]$
exists, where we allow ``undefined" as a possibility. Moreover,
for each $n$, if $\tilde{\imath}_n=\mu i (\exists j)[\psi(n, \omega
\cdot i +j) \conv]$ then $x^i_n$ is defined iff $i =
\tilde{\imath}_n$, and $x^{\tilde{\imath}_n}_n \in A \iff n \in B$.
Finally, for each $n$, the total number of times any $n$-marker
is defined or redefined, summing over all $i \leq i_n$, is at most $h(n)$.
\end{lemma}
\begin{proof}
An $n$-marker $x^i_n$ can only become defined (re-defined) via
step 2a of the construction. Thus at the stage when $x^i_n$ is
defined (re-defined), $i$ is least such that $\psi_s(n, \omega
\cdot i +j)\conv$. At the moment that $x^i_n$ is defined
(re-defined), any $x^k_n$ with $k>i$ that may have been defined
is undefined, and since $k>i$, will never be re-defined at a
later stage. That is, at any stage of the construction, there
is at most one $i$ with $x^i_n$ defined, and, as a function of
the stages, the index $i$ of the $n$-markers that are defined
is non-increasing. Since there is only one defined $n$-marker
at any given stage, the total number of times that an
$n$-marker is undefined by Step 1 of the construction is
bounded by $\Sigma^{g(n-1)}_{k=1}(\frac{k^2 -k}{2})$. Let
$\hat{h}(m)$ be the total number of stages where an $m$-marker
is defined (re-defined). A $0$-marker cannot be undefined by
step 1. In step 2, a $0$ marker can only be undefined if a new
$0$-marker, with lower index, is defined. Thus $\hat{h}(0) =
i_0=h(0)$. Similarly, $\hat{h}(n) =
\Sigma^{g(n-1)}_{k=1}(\frac{k^2 -k}{2}) + i_n +
\Sigma_{k=0}^{n}h(k) = h(n)$.

Finally, consider $x^{\tilde{\imath}_n}_n$. Let $s$ be a stage by
which all $m$-markers with $m \leq n$ have reached their
limits, and such that $A_t \rstrd x^{\tilde{\imath}_n}_n = A_s \rstrd
x^{\tilde{\imath}_n}_n$ for all $t \geq s$. Note that by definition
of $\tilde{\imath}_n$, we have that $x^{\tilde{\imath}_n}_n$ is defined
at stage $s$. Then by step 2b of the construction we have that
$x^{\tilde{\imath}_n}_n \in A$ iff $n \in B$.
\end{proof}

\begin{lemma} $B \leq_1 A^b$
\end{lemma}
\begin{proof}
We claim that $n \in B$ iff $g(n) \in A^b$. Consider the stage
$s$ when $x^{\tilde{\imath}_n}_n$ was defined for the last time. At
this stage, we set $\varphi_{k(n,r)}(g(n))\conv =
x^{\tilde{\imath}_n}_n$ for some $r$, and declare
$\Phi_{g(n)}^{\sigma \concat 1}(g(n)) \conv$ for all $\sigma$ of length
$x_n^{\tilde{\imath}} - 1$.  Since $k(n,r) < x^{\tilde{\imath}_n}_n$,
we have that if $x^{\tilde{\imath}_n}_n \in A$ then $g(n) \in A^b$.
Conversely, we only ever define $\Phi_{g(n)}(g(n))$ to halt in Step 2a
of the construction, and with an oracle that includes an $n$-marker.
Since all $n$-markers besides $x^{\tilde{\imath}_n}_n$ were
extracted from $A$ at stage $s$, we have that if
$x^{\tilde{\imath}_n}_n \nin A$ then $g(n) \nin A^b$. Now by the
previous lemma we have $x^{\tilde{\imath}_n}_n \in A$ iff $n \in B$,
so that $n \in B$ iff $g(n) \in A^b$ as desired.
\end{proof}

\begin{lemma}
$A^b \btl B$
\end{lemma}
\begin{proof}
Recall $x \in A^b \iff \exists e \leq x [\Phi_x^{A\smrstrd
\varphi_e(x)}(x) \conv]$. Recall also that $\es^b \btl B$. Let
$n$ be least such that $x < g(n)$.

Let $k$ be the total number of different oracles that appear to
witness $x \in A^b$ during the approximation of $A$. That is,
$k$ is maximal such that \begin{equation}\label{equ-search for
k} \exists s_1 ... \exists s_k \exists \sigma_1...\exists
\sigma_k (\sigma_i \not= \sigma_j \wedge \exists e \leq x
[\varphi_{e,s_i}(x) \conv \wedge \sigma_i = A_{s_i} \rstrd
\varphi_{e,s_i}(x)] \wedge \Phi_{x, s_i}^{\sigma_i} (x) \conv).
\end{equation}

According to step 1 of the construction, whenever some
$\varphi_e(x)\conv$ with $e \leq x$, all $m$-markers with
$m>g(n)$ are extracted from $A$. So, if $x \in A^b$, then the
only non-zero entries in the part of the oracle $A$ that is
used in the computation are those that arise from $m$-markers
with $m \leq n$. Since the total number of times $m$-markers
can be redefined is bounded by $h(m)$, and since each marker
can either be in or out of $A$, the number $k$ of possible
oracles is computably bounded (it is certainly bounded by
$2^{\Sigma_{l=0}^n h(l)}$). That is, we can bT compute $k$ from
$\es^b$ and hence $B$ using questions of the form
(\ref{equ-search for k}).

For each $m \leq n$, using at most $i_m$-many questions of the
form\\$(\exists x_1)...(\exists x_l) [x_{p+1} < x_p \wedge
(\exists j)\psi(m, \omega \cdot x_p +j ) \conv]$, we can bT
compute $\tilde{\imath}_m$ from $\es^b$ and hence $B$.

Similarly, we can $bT$ compute from $\es^b$, and hence from
$B$, the number of pairs $e \leq y \leq g(n)$ such that
$\varphi_e(y)\conv$. Thus we can $bT$ compute from $B$ the stage
$s$ by which point if $e \leq y \leq g(n)$ and $\varphi_e(y)
\conv$ then $\varphi_{e,s}(y) \conv$.

We can certainly $bT$-compute from $B$ the initial segment
$B\rstrd n$.

We now put the above facts together to compute whether $x \in
A^b$. If $k = 0$, then there is never any stage where it
appears that $x \in A^b$, so $x \nin A^b$. So suppose $k \not=
0$. Run the approximation of $A$ to find the $k$-many different
possible oracles which might witness $x \in A^b$. We know that
the only possible non-zero entries in the correct oracle come
from $x^{\tilde{\imath}_m}_m$ with $m \leq n$, and that
$x^{\tilde{\imath}_m}_m \in A$ iff $m \in B$. Now since we have
$bT$-computed from $B$ all the $\tilde{\imath}_m$ for $m \leq n$, we
can run the approximation of $A$ until the least stage $t$
greater than $s$ where markers of the form $x^{\tilde{\imath}_m}_m$
are defined for all $m \leq n$. The location of
$x^{\tilde{\imath}_m}_m$ at stage $t$ is its final location. Now,
using $B\rstrd n$, we have computed the true initial segment of
$A$ that is relevant for deciding whether $x \in A^b$. If this
oracle extends any of the $k$-many halting oracles that we
found, then $x \in A^b$. Otherwise, $x \nin A^b$.

\end{proof}

\begin{lemma} The functions $g$, $h$, and $k$ used in the construction exist. \label{ghlem} \end{lemma}
\begin{proof} Let $\Psi_{m,n,q}$ and $\Gamma_{n,q}$ denote the operations that are referred to in the main construction as $\varphi_{k(m,n)}$ and $\Phi_{g(n)}$, respectively, when the role of $g(n)$ in the construction (when not in the form $\Phi_{g(n)}$) is played by $\varphi_q (n)$.  We wish to find $g$, $h$, and $k$ which satisfy the less than and greater than constraints in the main proof, and a number $i$ such that $\varphi_{k(m,n)} = \Psi_{m,n,i}$, $\Phi_{g(n)} = \Gamma_{n,i}$, and $g(n) = \varphi_i (n)$ for all $m,n$.

By the padding lemma, for each $n,m,q$ let $K_{n,m,q}$ be an infinite, uniformly computable set such that for all $l \in K_{n,m,q}$ we have $\varphi_l = \Psi_{n,m,q}$.  Similarly, for all $n,q$ let $G_{n,q}$ be an infinite, uniformly computable set such that for all $l \in G_{n,q}$ we have $\Phi_l = \Gamma_{n,q}$.

We now define a uniformly computable procedure (in a parameter $q$) which we will label $\Theta_q$.  The procedure will use simultaneous induction to define three computable functions, $\tilde{g}$, $\tilde{h}$, and $\tilde{k}$.

We start the procedure by saying $\tilde{g}(-1) = -1$.  Given $\tilde{g}$ and $\tilde{h}$ up to $n-1$, we define $\tilde{h} (n)$ as we did in the main theorem, $\tilde{h}(n)= \Sigma^{n-1}_{t=0} \tilde{h}(t) + \Sigma^{\tilde{g}(n-1)}_{t=1}(\frac{t^2 -t}{2}) + i_{n}$.  Next, for each $m$ such that $0 \leq m < \tilde{h}(n)$ we assign the least possible element of $K_{n,m,q}$ as the value of $\tilde{k}(n,m)$ such that we satisfy $\tilde{g}(n-1) < \tilde{k}(n,0) < \tilde{k}(n,1) < \ldots < \tilde{k}(n,\tilde{h}(n) - 1)$.  Finally we assign the least element of $G_{n,q}$ bigger than $\tilde{k}(n,\tilde{h}(n) - 1)$ as the value of $\tilde{g}(n)$.  This completes our induction, and the procedure $\Theta_q$.

We note that if $\tilde{g}$, $\tilde{h}$, and $\tilde{k}$ come from procedure $\Theta_q$ then they meet the less than and greater than constraints in the main proof, and for all $m,n$ we have $\varphi_{\tilde{k}(n,m)} = \Psi_{n,m,q}$ and $\Phi_{\tilde{g}(n)} = \Gamma_{n,q}$.

Define a computable, injective function $w$ by letting $\varphi_{w(q)} (x) = \tilde{g} (x)$ where $\tilde{g}$ comes from procedure $\Theta_q$.  Let $i$ be given by the Recursion Theorem applied to $w$.  Finally, let $g$, $h$, and $k$ be given by $\tilde{g}$, $\tilde{h}$, and $\tilde{k}$ from procedure $\Theta_i$.  Then $\varphi_{k(m,n)} = \Psi_{m,n,i}$, $\Phi_{g(n)} = \Gamma_{n,i}$, and $\varphi_i = \varphi_{w(i)} = g$, as desired.
\end{proof}

\end{proof}

We note the proof above cannot be modified to find an $A$ such
that $A^b \leq_{tt} B$.

\section{Other jump operators}\label{sec:other jump operators}

In 1979, Gerla \cite{Gerla} proposed jump operators for the truth-table and bounded truth-table degrees.  We wish to compare his observations on these operators with some of the results shown so far for the bounded jump.  Since the original article is available only in Italian, we briefly summarize the definitions and highlight a few of the results from the paper.

We start with some basic definitions used in studying the truth-table degrees (see Rogers \cite{Rogers}).

\begin{definition} A $tt$-condition is a finite sequence $x_1 \ldots x_k \in \omega$ and a function $\alpha: 2^k \to 2$.  We say it is satisfied by $A$ if $\alpha(A(x_1) \ldots A(x_k)) = 1$.  We define $A^{tt} = \{x \setsep x$ is a $tt$-condition satisfied by $A \}$.
\end{definition}

We note that $A^{tt} \leq_{tt} A$ and $A \leq_1 A^{tt}$.  Gerla \cite{Gerla} uses $A^{tt}$ to define jumps $A_{tt}$ and $A_{bk}$ for the truth-table degrees and bounded truth-table degrees of norm $k$, respectively.

\begin{definition} $A_{tt} = \{ x \setsep \varphi_x (x) \conv \in A^{tt} \}$.  $A_{bk} = \{ x \setsep \varphi_x (x) \conv \in A^{tt}\ \wedge\ \varphi_x(x) \leq k \}$. \end{definition}

The behavior of $A_{tt}$ and $A_{bk}$ on the truth-table and bounded truth-table degrees shares several similarities with that of $A^\prime$ on the Turing degrees.  We state a few of the many results below.

\begin{theorem}[Gerla \cite{Gerla}] Let $k$ be a number and let $A$ and $B$ be sets.
\begin{enumerate}
\item{$A_{tt} \not\leq_{tt} A$.  $A_{bk} \not\leq_{bk} A$.}
\item{$A \leq_{tt} B \Rightarrow A_{tt} \leq_1 B_{tt}$.}
\item{$A <_1 A_{bk} \leq_1 A_{b(k+1)} \leq_1 A_{tt} \leq_1 A^\prime$.}
\item{$\es_{bk} \equiv_1 \es_{tt} \equiv_1 \es^\prime$.}
\end{enumerate}
\end{theorem}

We demonstrated earlier the connection between the bounded jump and the Ershov hierarchy.  We see that the finite levels of the Ershov hierarchy share a similar (but weaker) connection with $A_{bk}$.

\begin{theorem}[Gerla \cite{Gerla}] Let $A$ be $n$-c.e.\ and let $B \leq_1 A_{bk}$.  Then $B$ is $(nk + 1)$-c.e. \end{theorem}

Let $\es_{n(bk)}$ denote the $n$-th iteration of the $bk$ jump of the empty set.  It follows from the theorem that if $A \leq_1 \es_{n(bk)}$ then $A$ is $(1 + k + k^2 + \ldots k^{n-1})$-c.e.\ \cite{Gerla}.

Since Gerla's truth-table jump is designed for a stronger reducibility, we expect it to be weaker than the bounded jump.  We prove that for every set $A$ we have $A_{tt} \leq_1 A^b$, but there are many sets $X$ such that $X^b \not\leq_{bT} X_{tt}$.

\begin{prop} $A_{tt} \leq_1 A^{b_0}$ \end{prop}
\begin{proof}
Let $f$ and $\Phi_k$ witness that $A^{tt} \leq_{bT} A$.  We define computable, injective functions $h$ and $j$.  Let $\varphi_{h(e)}(z) = f(\varphi_e(e))$.

\ \\
Define $j$ by $\Phi^C_{j(e)} (z) \conv$ iff $\varphi_e (e) \conv$ and $\varphi_e (e) \in \Phi_k^C$.

\ \\
Let $z$ represent an arbitrary dummy variable.  We note the following.

\begin{align*}
x \in A_{tt}
\iff \ & \varphi_x (x) \conv \in A^{tt} \\
 \iff \ &
\varphi_x(x) \conv \in \Phi_k^A \\
\iff \ & \varphi_x(x) \conv \in
\Phi_k^{A \smrstrd f(\varphi_x (x))} \\
\iff \ & \varphi_x(x) \conv \in \Phi_k^{A \smrstrd \varphi_{h(x)}(z)} \\
\iff \ & \Phi_{j(x)}^{A \smrstrd \varphi_{h(x)}(z)}(z) \conv \\
\iff \ &
\langle j(x), h(x), z \rangle \in A^{b_0}
\end{align*}

Thus $A_{tt} \leq_1 A^{b_0}$.

\end{proof}

\begin{cor} $A_{tt} \leq_1 A^b$ \end{cor}

\begin{theorem} There is a c.e.\ set $A$ such that $A^b \not\leq_{bT} A_{tt}$. \end{theorem}

\begin{proof} First note that if we have a computable approximation to a set $A$, then this
induces an obvious approximation for $A_{tt}$. Namely, if
$\varphi_{x,s}(x) \up$ then $x \nin A_{tt} [s]$, and if
$\varphi_{x,s}(x) \dn$ then $x \in A_{tt} [s] \iff \varphi_{x}
(x) \in A[s]^{tt}$.  We also have an approximation for $A^b$ by
$x \in A^b[s] \iff \exists i \leq x[\varphi_{i,s}(x) \conv\ \wedge\
\Phi_{x}^{A\smrstrd \varphi_{i,s} (x)} (x) \conv[s]]$. We note
that if $A$ is c.e.\ then these are both $\Delta^0_2$
approximations.

For $k \in \omega$, let $l(k,s) = \max\{\varphi_{x,s}(x) \mid x
\leq k \ \wedge\ \varphi_{x,s} (x) \conv \}$. Note that for $s <t$, if $A\Eres l(k,s) [s] =
A\Eres l(k,s) [t]$, then $A_{tt} \Eres k [s] = A_{tt} \Eres k
[t]$ \emph{unless} $\varphi_{x,t}(x) \dn$ for some $x \leq k$
such that $\varphi_{x,s}(x) \up$.

We now proceed with the construction of $A$. We must meet for all $n \in \omega$ the requirement
$$R_n: (\neg \forall x) [\varphi_{\pi_2(n)}(x) \conv \ \wedge\ \Phi_{\pi_1(n)}^{A_{tt} \Eres \varphi_{\pi_2(n)}(x)}(x) = A^b(x)]$$
where $\pi_1$ and $\pi_2$ are projection functions for some canonical pairing function.  

To ease notation, we will use the following convention.  We write $\Phi_i$ for $\Phi_{\pi_1(i)}$.  For $\varphi$ we distinguish between two cases.  We write $\varphi_i (x_i)$ for $\varphi_{\pi_2(i)} (x_i)$.  However $\varphi_y (y)$ maintains its usual meaning for any $y$.

Let $e_0 < e_1 <e_2 < \ldots$ be a computable list such that we
control $\varphi_{e_i}$ and $\Phi_{e_i}$ by the Recursion
Theorem.  A formal definition can be accomplished by the methods 
used in lemma \ref{ghlem}.  

We will use a set of movable markers $x_i$ for $i
\num$ such that for all $i$ we have $x_i = e_j$ for some $j$.
We will also make use of a restraint function $r$.  We start with
$r(n)[0] = 0$ for all $n$.

\emph{Stage 0:} Let $x_0 = e_0$.

\emph{Stage s+1:} For each $m$ let $r(m)[s+1] = \max
\{\varphi_{x,s+1}(x) \mid x \leq \varphi_l(x_l) [s]$ for
some $l <m \}$ (we say $r(m)[s+1]=0$ if this set is empty).
Let $k$ be least such that $r(k) [s+1]
> r(k) [s]$ (if no such $k$ exists, use $k=s$).
Undefine all $x_m$ with $m \geq k$.

\emph{Case 1:} There is no $n < k$ such that $x_n$ is defined and $\Phi_n^{A_{tt} \Eres \varphi_n(x_n)}(x_n) [s] \dn =
A^b (x_n) [s]$.

We then let $m$ be least such that $x_m$ is not defined, and define $x_m$ to be the least $e_i$ that has
not been used in the construction (proceed to the next stage).

\emph{Case 2:} Else.

We then let $n < k$ be least such that $x_n$ is defined and $\Phi_n^{A_{tt} \Eres \varphi_n(x_n)}(x_n) [s] \dn =
A^b (x_n) [s]$. Undefine all $x_m$ with $m >n$.  If it has not yet been defined (with the current value of $x_n$),
let $\varphi_{x_n}(x_n) =r(n)[s+1] + \max\{A_s\} + \varphi_n(x_n)$.

\emph{Subcase 2A:} $A^b (x_n) [s] = 0$.

Set $\Phi_{x_n}^{A_s \Eres \varphi_{x_n}(x_n)}(x_n)\dn$, so that $A^b(x_n) [s+1] =1$.

\emph{Subcase 2B:} $A^b(x_n) [s] = 1$.

Choose the least $x >r(n)$ such that $x \nin A [s]$, and enumerate $x \in A [s+1]$. We demonstrate later
that we have $x \leq \varphi_{x_n}(x_n)$, so that this will cause $A^b(x_n) [s+1] = 0$.

This completes the construction of $A$.

It is easy to see that the set constructed is c.e.  We claim
that for each $n$, $x_n = \lim_s x_n [s]$ exists, and provides
a witness for $R_n$.  We say that a requirement $R_n$ receives
attention if we perform case 2 of the construction for some $x_n$.

\begin{lemma}
For each $n$, $x_n = \lim_s x_n [s]$ exists, and provides a
witness for $R_n$. Moreover, the requirement $R_n$ receives
attention at most finitely often.
\end{lemma}
\begin{proof} Since $x_0$ is never undefined, it reaches its limit at stage
$0$. Assume that $x_{l}$ with $l < m$ have reached their limit,
and if $\varphi_{l} (x_{l}) \conv$ then it has already done so.
Then the value $r(m)$ can increase at most $\max \{\varphi_{l}
(x_{l}) \mid l < m \}$-many more times, and so there must be a
stage after which $x_m$ is never undefined.

Assume for a contradiction that $x_n$ is least such that
$\Phi_n^{A_{tt} \Eres \varphi_n(x_n)}(x_n) \dn = A^b (x_n)$.
Let $s$ be the least stage after which no $x_m$ with $m <n$
requires attention, and $r(n)$ has reached its limit. So at
stage $s+1$ of the construction, $x_n$ is defined, and is never
again undefined. Since $\Phi_n^{A_{tt} \Eres
\varphi_n(x_n)}(x_n) \dn = A^b (x_n)$, there is a least stage
$s_0 > s+1$ such that $\Phi_n^{A_{tt} \Eres
\varphi_n(x_n)}(x_n) [s_0] \dn = A^b (x_n)[s_0]$. Since $s_0$ is the
first stage where $R_n$ requires attention with this value of
$x_n$, we define $\varphi_{x_n}(x_n) = r(n) + \max\{A_{s_0-1}\} +
\varphi_n(x_n)$ at stage $s_0$, and we have $A^b(x_n) [s_0 -
1] = 0$. So at stage $s_0$ of the construction, we set
$\Phi_{x_n}^{A_{s_0} \Eres \varphi_{x_n}(x_n)}(x_n)\dn$, so
that $A^b(x_n) [s_0] = 1$. Note that at stage $s_0$ there are at
least $\varphi_n(x_n)$-many numbers greater than $r(n)$ and
less than $\varphi_{x_n}(x_n)$ available to enumerate into $A$.

Let $s_0 < s_1 <s_2 < \ldots$ be all the further stages of the
construction where $R_n$ receives attention.  We will show that
for all even $k$ an element is enumerated into
$\es^\prime \rstrd \varphi_n (x_n)$ at some stage $t$ with
$s_k < t \leq s_{k+1}$.  It follows that the sequence $s_0, s_1, \ldots$
must be finite, contradicting the assumption that
$\Phi_n^{A_{tt} \Eres \varphi_n(x_n)}(x_n) \dn = A^b (x_n)$.  We
will also show inductively that there is sufficient room to
enumerate elements into $A$ between $r(n)$ and $\varphi_{x_n} (x_n)$,
as claimed earlier.

Let $k$ be even, and assume for our induction that there are at least
$(\varphi_n(x_n) - \frac{k}{2})$-many numbers greater than
$r(n)$ and less than $\varphi_{x_n}(x_n)$ available to
enumerate into $A$.  Without loss of generality, suppose that at stage
$s_k$ we ensured $A^b(x_n) [s_k] = 1$.  Since all requirements $R_m$ with
$m<n$ have stopped acting, no requirement $R_m$ with $m \leq n$
enumerated into $A$ at any stage $s_k \leq t \leq s_{k+1}$.
Furthermore, since $r(m)\geq \varphi_{x_n}(x_n)$ for all $m
>n$, no requirement $R_m$ enumerates into
$A\Eres \varphi_{x_n}(x_n)$ at any stage $s_k \leq t <s_{k+1}$.
Hence $A^b(x_n) [s_{k+1} -1] = 1$ and $\Phi_n^{A_{tt} \Eres
\varphi_n(x_n)}(x_n) [s_{k+1}-1] = 1$. So $A_{tt} \Eres
\varphi_n(x_n)[s_k -1] \not= A_{tt} \Eres \varphi_n(x_n)[s_{k+1}-1]$.

Using our observation from the start of the proof of the theorem, to
demonstrate that there is a $y \leq \varphi_n(x_n)$ such that
$\varphi_{y, s_k}(y) \up$ but $\varphi_{y, s_{k+1}}(y) \dn$, it suffices
to show $A \rstrd l(\varphi_n(x_n),s_k - 1)[s_k - 1] = A \rstrd
l(\varphi_n(x_n,s_{k+1}-1)[s_{k+1}-1]$.  Between stages $s_k - 1$ and
$s_{k+1} - 1$, the construction only runs subcase 2 for a requirement
$R_m$ with $m > n$.  Hence no element is enumerated into $A \rstrd r(n+1)$.
Since $r(n+1) \geq l(\varphi_n(x_n),s_k - 1)$ we have
$A \rstrd l(\varphi_n(x_n),s_k - 1)[s_k - 1] = A \rstrd
l(\varphi_n(x_n,s_{k+1}-1)[s_{k+1}-1]$, as desired.  Thus some $y \leq \varphi_n (x_n)$
was added to $\es^\prime$ between stages $s_k$ and $s_{k+1}$.

At stage $s_{k+1}$ the least $x > r(n)$ such that $x \nin A [s_{k+1}-1]$, was
enumerated into $A [s_{k+1}]$.  By induction hypothesis, we had
$x \leq \varphi_{x_n}(x_n)$, so that $A^b(x_n)[s_{k+1}]= 0$. Note that
at stage $s_{k+1} +1$ there are at least $(\varphi_n(x_n) - \frac{k}{2} -1)$-many
numbers greater than $r(n)$ and less than $\varphi_{x_n}(x_n)$
available to enumerate into $A$.

At stage $s_{k+2}$, we acted because $\Phi_n^{A_{tt} \Eres
\varphi_n(x_n)}(x_n) [s_{k+2} -1] \dn = A^b (x_n) [s_{k+2} -1] = 0$. We set
$\Phi_{x_n}^{A_{s_{k+2}} \Eres \varphi_{x_n}(x_n)}(x_n)\dn$, so
that $A^b(x_n) [s_{k+2}] = 1$. There was no enumeration into $A$
below $r(n+1)$ at any stage $s_{k+1} <t \leq s_{k+2}$, so that at
stage $s_{k+2} +1$ there are at least $(\varphi_n(x_n) -
\frac{k}{2} -1)$-many numbers greater than $r(n)$ and less than
$\varphi_{x_n}(x_n)$ available to enumerate into $A$.  This
completes our induction.

Since we can only reach a stage $s_k$ with $k$ even if a number
less than $\varphi_n(x_n)$ enters $\es^\prime$, and since we
have left room to enumerate into $A$ in the desired interval at
each such stage, it follows that there can be only finitely
many stages $s_k$, as desired.
\end{proof}
\end{proof}

A similar proof can be used to show that every 2-generic $A$ is such that $A^b \not\leq_{bT} A_{tt}$.

Finally, we note that the minijump operator developed by Ershov \cite{minijump} works on the $pm$ degrees in a manner similar to $A_{tt}$ on the truth-table degrees (See Odifreddi \cite{OdifreddiCRT}, Volume II page 732).

\section{Further Study}

There is considerable room left to explore in the study of the bounded jump.  We can examine to what degree do important results for the Turing jump on the Turing degrees also hold for the bounded jump on the bounded Turing degrees, particularly in cases where these results do not hold for the Turing jump on the bounded Turing degrees.

For example, Sacks \cite{SacksJI} proved that for every $\Sigma_2$ set $X \geq_T \es^\prime$ there is a c.e.\ set $Y$ such that $Y^\prime \equiv_T X$.  Csima, Downey, and Ng \cite{CsimaDowneyNg} proved that Sacks jump inversion does not hold for the Turing jump on the bounded Turing degrees.  We do not know if Theorem \ref{SchThm} holds if we add the requirement that $Y$ is c.e.

We can also look at concepts related to the Turing jump.  We say that a set $X$ is bounded-high if $X^b \geq_{bT} \es^{2b}$ and bounded-low if $X^b \leq_{bT} \es^b$.  We can then attempt to characterize which sets are bounded-high and bounded-low.  Finally, the jumps for the truth-table and bounded truth-table degrees developed by Gerla \cite{Gerla} could be considered in more detail.

\nocite{*}
\bibliography{BoundedJump}
\end{document}